\documentclass[12pt, twoside, leqno]{article}

\usepackage{amsmath,amsthm}
\usepackage{amssymb}
\usepackage{hyperref}

\usepackage{enumerate}

\usepackage{graphicx}

\newtheorem*{thm*}{Theorem}
\newtheorem*{cor*}{Corollary}
\newtheorem*{lem*}{Lemma}
\newtheorem*{def*}{Definition}
\newcommand{\SL}{\operatorname{SL}}
\newcommand{\GL}{\operatorname{GL}}
\newcommand{\SG}{\operatorname{SG}}

\newtheorem{lemma}{Lemma}[section]
\newtheorem{theorem}{Theorem}[section]

\newtheorem*{question*}{Question}

\theoremstyle{definition}

\numberwithin{equation}{section}

\begin{document}


\baselineskip=17pt


\title{Nonexistence of Smooth Effective One Fixed Point Actions of Finite Oliver Groups on Low-dimensional Spheres}

\author{Agnieszka Borowiecka\\
E-mail: aborowiecka@wp.pl
\and 
Piotr Mizerka\\
Faculty of Mathematics and Computer Science\\ 
Adam Mickiewicz University\\
61-614 Poznań\'n, Poland\\
E-mail: piotr.mizerka@amu.edu.pl}

\date{}

\maketitle


\renewcommand{\thefootnote}{}


\footnote{\emph{Mathematical discipline of the article}: Manifolds and cell complexes.}


\renewcommand{\thefootnote}{\arabic{footnote}}
\setcounter{footnote}{0}


\begin{abstract}
According to \cite{Morimoto1998}, a finite group $G$ has a smooth effective one fixed point action on some sphere if and only if $G$ is an Oliver group. For some finite Oliver groups $G$ of order up to $216$, and for $G=A_5\times C_p$, where $p=3,5,7$, we present a strategy of excluding of smooth effective one fixed point $G$-actions on low-dimensional spheres.
\end{abstract}
\section{Introduction}
The paper is organized as follows. In Introduction we give a historical overview and state the main result of the paper. We provide definitions necessary to understand this result. In the next sections we focus on theoretical background. Then, we apply it to an algorithm of excluding of postulated actions. Throughout, unless otherwise specified, all groups are assumed to be finite.

In the paper by an \emph{\textbf{action}} of a group $G$ on a smooth manifold we mean any homomorphism $\ast $ of the group $G$ into group of all diffeomorphisms of $M$. For $g\in G$ and $x\in M$ we will write $g*x$ instead of formally more corrected $\ast(g)(x)$.
The \emph{\textbf{fixed point set}} of such action is the set
	$$
	M^G=\{x\in M:g*x=x \text{ for any } g\in G\}
	$$
In the case when the point fixed set consists of exactly one point we call the action to be \emph{\textbf{one fixed point}} and in the case this set is empty we call the action to be \emph{\textbf {fixed point free}}. An action of a group $G$ on a smooth manifold $M$ is called \emph{\textbf{effective}} if it is a monomorphism.
\par
The first construction of a one fixed point action of a group $G$ on a sphere was obtained by Stein \cite{Stein1977} for $\SL_2(\mathbb{F}_5)$ - the group of $2\times 2$ matrices with determinant $1$ and entries in the field $\mathbb{F}_5$ of five elements. Laitinen, Morimoto and Pawa\l{}owski (\cite{Morimoto1998},\cite{MorimotoPawalowski1995}) proved that a group $G$ admits an effective one fixed point action on some sphere if $G$ is an \emph{\textbf{Oliver group}} which  means that G does not contain a normal subgroup $H$ and a subgroup $P$ such that $P$ is a normal subgroup of $H$, $P$ and $G/H$ are groups of prime power orders and the group $H/P$ is cyclic. Earlier, it was proved by Oliver \cite{Oliver1996} that the above property characterizes groups which have  fixed point free action on a disc.
Using the Slice Theorem (\cite{Audin2004}, Theorem I.2.1), we can transform one fixed point actions on spheres into fixed point free actions on disks by cutting out the equivariant neighbourhood of the fixed point. Therefore, a group $G$ acts with one fixed point on some sphere if and only if it is an Oliver group. It is known that this statement holds also in the case of effective actions. What is the lowest dimension of such a sphere is still unknown for most of Oliver groups. Morimoto \cite{Morimoto1987}, Furuta \cite{Furuta1989}, Buchdahl, Kwasik and Schultz \cite{Kwasik1990} established this dimension to be at least $6$ (in fact, this is the case for the alternating group $A_5$ ).\par
The second author of this paper developed in \cite{Borowiecka2016} methods allowing to exclude one fixed point actions of $\SL_2(\mathbb{F}_5)$ on the sphere of dimension $8$. Applying computations from GAP \cite{GAP4}, we extend these methods to some Oliver groups $G$ of order up to $420$. 
\par 
The following theorem is the main result of this paper.
\begin{theorem}
	\emph{Using the notations of groups from GAP \cite{GAP4} the following  Oliver groups have no effective one fixed point action on $S^n$:}
	\begin{itemize}
		\item $n=6,7,8,10$ \emph{and} $G\in\{\SG(216,90)$, $\SG(216,92)$, $\SG(216,95)\}$, $\SG(216,96)\}$,
		\item $n=6,7,8$ \emph{and} $G\in\{\SG(144,121)$, $\SG(144,124)\}$,
		\item $n=6,7,9$ \emph{and} $G=\SG(126,9)$,
		\item $n=6,7$ \emph{and} $G=\SG(168,49)$,
		\item $n=6,8$ \emph{and} $G\in\{\SL_2(\mathbb{F}_5)$, $\SG(216,91)$, $\SG(216,94)\}$,
		\item $n=6$ \emph{and} $G\in\{A_5\times C_3$, $A_5\times C_5$, $A_5\times C_7$, $S_4\times C_3$, $S_4\times C_6$, $S_4\times C_9$, $S_3\times A_4$, $S_3\times A_4\times C_2$, $S_3\times S_4$, $\SG(72,43)$, $\SG(144,123)$, $\SG(144,126)$, 
		$\SG(144,129)$, $\SG(144,189)$, $\SG(216,93)\}$, 
		\item $n=7$ \emph{and} $G\in\{S_5$, $\GL_2(\mathbb{F})_3\times C_3$, $\SL_2(\mathbb{F}_3)\times S_6$, $\SG(144,125)$, $\SG(144,127)\}$,
		\item $n=7,9,10,11,17$ \emph{and} $G=\GL_3(\mathbb{F}_2)$.
	\end{itemize}
\end{theorem}
\section{Discrete submodules restriction}
In the sequel all undefined
terms and notations are well known to mathematicians familiar with the theory of group
actions on manifolds. They can be found in \cite{Bredon1972}.

In this section we present a strategy of excluding effective one fixed point actions of groups on spheres which concerns the case when the fixed point sets of actions of some particular subgroups are zero-dimensional. This justifies the name "discrete".

Let us introduce, after Morimoto and Tamura \cite{MorimotoTamura2018}, the following notation. 
\begin{itemize}
	\item $\mathcal{G}_p^q$ - the set of all groups $G$ for which there exists subgroups $P,H\leq G$ such that $P\trianglelefteq H\trianglelefteq G$, $P$ is a $p$-group, $G/H$ is a $q$-group and $H/P$ is cyclic for some primes $p$ and $q$.
	\item $\mathcal{G}=\cup_{p,q}\mathcal{G}_p^q$
	\item $\mathcal{G}_p^q(G)$ - the intersection of $\mathcal{G}_p^q$ with the set of all subgroups of $G$
	\item $\mathcal{G}(G)$ - the intersection of $\mathcal{G}$ with the set of all subgroups of $G$.
\end{itemize}
Then we have a very useful lemma from the work of Morimoto and Tamura.
\begin{lemma}\emph{\cite{MorimotoTamura2018}\label{l:1}
Let $\Sigma$ be a $\mathbb{Z}$-homology sphere with $G$-action and $x_0\in\Sigma^G$. Then we have
\begin{enumerate}[(1)]
	\item For every $H\in\bigcup_p\mathcal{G}_p^1(G)$, $\chi(\Sigma^H)=1+(-1)^{\dim T_{x_0}(\Sigma^H)}$.
	\item  For every $H\in\bigcup_p\mathcal{G}_p^q(G)$, $\chi(\Sigma^H)\equiv 1+(-1)^{\dim T_{x_0}(\Sigma^H)}$ $(\mod q)$.
\end{enumerate}}
\end{lemma}
Using Lemma \ref{l:1}, we are able to figure out a strategy of excluding one fixed point actions, given by the following theorem.
\begin{theorem}\label{strategy:1}
\emph{	Let a group $G$ act on a $\mathbb{Z}$-homology sphere $\Sigma$. Suppose there exist subgroups $H_1,H_2\leq \mathcal{G}(G)$ and a primary subgroup $P\leq H_1\cap H_2$. Assume $G=\langle H_1\cup H_2\rangle$ and $\dim T_x\Sigma^P=0$ for some $x\in \Sigma^G$. Then $\Sigma^G$ cannot consist of exactly one point.}
\end{theorem}
\begin{proof}
	The order of $P$ is a power of some prime $p$. From the Smith theory we deduce, that $\Sigma^P$ is a $\mathbb{Z}_p$-homology sphere. Since $\dim T_x\Sigma^P=0$, we conclude that it consists of two points. As $\Sigma^{H_1}$ and $\Sigma^{H_2}$ are contained in $\Sigma^P$, it follows that $\Sigma^{H_1}$ and $\Sigma^{H_2}$ both consist of at most two points. In particular the Euler characteristics of $\Sigma^{H_1}$ and $\Sigma^{H_2}$ are equal to their cardinalities.
	
	Assume that $\Sigma^G=\text{p}$. The Euler characteristics of $\Sigma^{H_1}$ and $\Sigma^{H_2}$ both are greater or equal $1$. Moreover, from Lemma \ref{l:1}, we deduce $\chi(\Sigma^{H_1}),\chi(\Sigma^{H_2})\in\{0,2\}$. Since these characteristics are positive, we conclude that
	$$
	\chi(\Sigma^{H_1})=\chi(\Sigma^{H_2})=2.
	$$
	This means $\Sigma^{H_1}$, $\Sigma^{H_2}$ and $\Sigma^P$ consist of two points, which in connection with $\Sigma^{H_1},\Sigma^{H_2}\subseteq\Sigma^P$, gives
	$$
	\Sigma^P=\Sigma^{H_1}=\Sigma^{H_2}.
	$$
	However, since $\langle H_1\cup H_2\rangle=G$,
	$$
	\text{p}=\Sigma^G=\Sigma^{H_1}\cap\Sigma^{H_2}=\Sigma^P,
	$$
	which is a contradiction, for $\Sigma^P$ is a two-point set.
\end{proof}
\section{Intersection number restriction}
\label{s:3}
Following \cite{Davis2001}, suppose $M$ is a smooth, compact, oriented manifold of dimension $m$ and $A, B$ are its smooth, compact, connected and oriented submanifolds of dimensions $a$ and $b$ respectively such that $m=a+b$. Moreover, assume $A$ and $B$ are transverse in $M$. Choose a point $x\in A\cap B$. Since $A$ and $B$ are transverse and of complementary dimensions in $M$, we can look at the tangent space to $M$ at the point $x$ in two ways:
\begin{itemize}
	\item as $T_xM$ with basis induced from the orientation of $M$,
	\item as $T_xA\oplus T_xB$ with basis given by the bases of $T_xA$ and $T_xB$ induced by the orientations of $A$ and $B$.
\end{itemize}
Denote by $\eta_x$ the sign of the determinant of change-of-base matrix from $T_xM$ to $T_xA\oplus T_xB$. The \emph{\textbf{intersection number}} of $A$ and $B$ in $M$ is then a value $\sum_{x\in A\cap B}\eta_x$, denoted by $A\cdot B$.
Now, consider the fundamental classes $[A]\in H_a(A)$ and $[B]\in H_b(B)$. They induce elements $(i_A)_*([A])\in H_a(M)$ and $(i_B)_*([B])\in H_b(M)$, where $i_A:A\subset M$ and $i_B:B\subset M$ are inclusions.
The intersection number reveals a very interesting algebraic property which we shall use later.
\begin{theorem}\emph{\cite{Davis2001}\label{t:7}
	For $A, B$ and $M$ are as above, we have
	$$
	A\cdot B=\langle\alpha\cup\beta,[M]\rangle,
	$$
	where $\alpha$, $\beta$ are Poincare duals to $(i_A)_*([A])$ and $(i_B)_*([B])$ respectively, $[M]$ is the fundamental class of $M$ and $\langle,\rangle$ is the Kronecker pairing.}
\end{theorem}
We will focus now on ensuring the conditions which allow us to define the intersection number for our particular purposes. 

The following lemma provides conditions ensuring transversality.
\begin{lemma}\label{l:4}
\emph{	Assume a group $G$ acts on a smooth manifold $M$ with $M^G$ connected. Suppose there exist subgroups $H_1, H_2\leq G$ and $H\leq H_1\cap H_2$ such that the following conditions hold:
	\begin{enumerate}
		\item $\langle H_1\cup H_2\rangle=G$,
		\item 
		$$
		\dim C(H_1)+\dim C(H_2)-\dim M^G=\dim C(H),
		$$
		where, for a given subgroup $K\leq G$, $C(K)$ stands for the connected component of $M^K$ containing $M^G$.
	\end{enumerate}
Then $M^{H_1}$ and $M^{H_2}$ are transverse in $M^H$.}
\end{lemma}
\begin{proof}
	Pick $x\in M^{H_1}\cap M^{H_2}\overset{\mathrm{\langle H_1\cup H_2\rangle=G}}{=}M^G$. It is enough to show that
	\begin{equation}\label{e:11}
	\dim T_xM^{H_1}+\dim T_xM^{H_2}-\dim(T_xM^{H_1}\cap T_xM^{H_2})=\dim T_xM^H
	\end{equation}
	From the dimension assumption we have
	$$
	\dim T_xM^{H_1}+\dim T_xM^{H_2}-\dim T_x(M^{H_1}\cap M^{H_2})=\dim T_xM^H
	$$
	If we prove that
	\begin{equation}\label{e:2}
	T_x(M^{H_1}\cap M^{H_2})=T_xM^{H_1}\cap T_xM^{H_2}
	\end{equation}
	we get (\ref{e:11}).
	
	From the Slice Theorem there exists an open neighborhood $U$ of $x$ and a $G$-diffeomorphism $f:U\longrightarrow T_xM$.
	
	Let $K\leq G$. We show that
	\begin{equation}\label{e:3}
	f(U^K)=T_xM^K
	\end{equation}
	Since $f$, as a $G$-diffeomorphism, preserves fixed point sets, we have $f(U^K)=(T_xM)^K$. As the action of $K$ on $T_xM$ is linear, it follows that $(T_xM)^K$ is a linear subspace of $T_xM$ and therefore it is connected. Hence $U^K$ is connected. One can consider then the dimensions of $U^K$ and $(T_xM)^K$. The former if equal to $\dim C(K)$. As $f$ is a $G$-diffeomorphism, it preserves the dimension, hence
	$$
	\dim(T_xM)^K=\dim f(U^K)=\dim U^K=\dim C(K).
	$$
	On the other hand, $\dim T_xM^K=\dim C(K)$. Therefore we see that $T_xM^K$ and $(T_xM)^K$ are linear subspaces of $T_xM$ of the same dimension. So, to prove $T_xM^K=(T_xM)^K$, it suffices to show $T_xM^K\subseteq(T_xM)^K$.
	
	Let $[\gamma]\in T_xM^K$, where $\gamma:I\longrightarrow M^K$ is a curve in $M^K$. Pick $k\in K$. We have for $t\in I$
	$$
	k\gamma(t)=\gamma(t)
	$$
	since $\gamma$ is invariant under the action of $K$. Hence $k\gamma=\gamma$ and therefore $k[\gamma]=[k\gamma]=[\gamma]$. We conclude then that $[\gamma]\in (T_xM)^K$ and $T_xM^K\subseteq(T_xM)^K$ and, as a consequence, $T_xM^K=(T_xM)^K$. Hence, indeed
	$$
	f(U^K)=T_xM^K.
	$$
	Using (\ref{e:3}) and the fact that $f$ is injective, we have
	\begin{align*}
	T_xM^G&=f(U^G)=f(M^G\cap U)=f(M^{H_1}\cap M^{H_2}\cap U)\\
	&=f((M^{H_1}\cap U)\cap(M^{H_2}\cap U))=f(U^{H_1}\cap U^{H_2})\\
	&=f(U^{H_1})\cap f(U^{H_2})=T_xM^{H_1}\cap T_xM^{H_2}.
	\end{align*}
	On the other hand
	$$
	T_xM^G=T_x(M^{H_1}\cap M^{H_2}),
	$$
	which completes the proof of (\ref{e:2}).
\end{proof}
Let us care about the orientability assumption.
\begin{lemma}\label{l:2}
\emph{	Any $\mathbb{Z}_p$-homology sphere is orientable.}
\end{lemma}
\begin{proof}
	Suppose $\Sigma$ is a $\mathbb{Z}_p$-homology sphere which is not orientable. This means $H_n(\Sigma)=0$, where $n=\dim\Sigma$. Using the Universal Coefficient Theorem, we have for $k=1,\ldots,n$ the following exact sequence
	$$
	0\rightarrow H_k(\Sigma)\otimes\mathbb{Z}_p\rightarrow H_k(\Sigma;\mathbb{Z}_p)\rightarrow\text{Tor}(H_{k-1}(\Sigma),\mathbb{Z}_p)\rightarrow 0
	$$
	Since $H_k(\Sigma;\mathbb{Z}_p)$ vanish for $k=1,\ldots,n-1$, we have 
	$$
	H_k(\Sigma)\otimes\mathbb{Z}_p=\text{Tor}(H_{k-1}(\Sigma),\mathbb{Z}_p)=0
	$$ 
	and consequently 
	\begin{equation}\label{e:1}
	H_k(\Sigma)\cong\mathbb{Z}_{q_{k,1}}\oplus\ldots\oplus\mathbb{Z}_{q_{k,l_k}} 
	\end{equation}
	where $\gcd(q_{k,i},p)=1$.\par
	For $k=n$, we have $H_n(\Sigma)=0$ by our assumption. Hence, the Universal Coefficient Theorem gives in this case the exactness of
	$$
	0\longrightarrow H_n(\Sigma;\mathbb{Z}_p)\longrightarrow\text{Tor}(H_{n-1}(\Sigma),\mathbb{Z}_p)\longrightarrow 0
	$$
	Since $H_n(\Sigma;\mathbb{Z}_p)\cong\mathbb{Z}_p$, we conclude that
	$$
	\text{Tor}(H_{n-1}(\Sigma),\mathbb{Z}_p)\cong\mathbb{Z}_p
	$$
	On the other hand, by (\ref{e:1}),
	\begin{align*}
	\text{Tor}(H_{n-1}(\Sigma),\mathbb{Z}_p)&\cong\text{Tor}(\mathbb{Z}_{q_{n-1,1}}\oplus\ldots\oplus\mathbb{Z}_{q_{n-1,l_{n-1}}},\mathbb{Z}_p)\\
	&\cong\text{Tor}(\mathbb{Z}_{q_{n-1,1}},\mathbb{Z}_p)\oplus\ldots\oplus\text{Tor}(\mathbb{Z}_{q_{n-1,l_{n-1}}},\mathbb{Z}_p)=0.
	\end{align*}
\end{proof}
We are ready now to establish the intersection number strategy of excluding one fixed point group actions on spheres.
\begin{theorem}\label{strategy:2}
\emph{	Let a group $G$ act on a $\mathbb{Z}$-homology sphere $\Sigma$ with $\Sigma^G$ connected. Assume there exist subgroups $H_1,H_2\leq G$ with $\langle H_1\cup H_2\rangle=G$ such that $C(H_i)$ are of positive dimension for $i=1,2$. Suppose there exists a $p$-group $P\leq H_1\cap H_2$ such that
	$$
	\dim C(H_1)+\dim C(H_2)=\dim\Sigma^P.
	$$
	and one of the following is true:
	\begin{enumerate}
		\item the orders of $H_1$ and $H_2$ are odd,
		\item $P$ is normal in $H_1$ and $H_2$ and the orders of quotient groups $H_1/P$ and $H_2/P$ are odd.
	\end{enumerate}
	Then $\Sigma^G$ is not a one point set.}
\end{theorem}
\begin{proof}
	Assume $\Sigma^G$ consists of one point. We conclude from Lemma \ref{l:4} that $\Sigma^{H_1}$ and $\Sigma^{H_2}$ are transverse in $\Sigma^P$. Now, we show that $\Sigma^P$ is an orientable manifold. We know from Smith theory, that it is a $\mathbb{Z}_p$-homology sphere. Hence, by Lemma \ref{l:2}, we conclude that $\Sigma^P$ is orientable. Having this, notice that $\Sigma^{H_1}$ and $\Sigma^{H_2}$ are orientable too. Indeed, if $H_1$ and $H_2$ are of odd order, this follows directly from \cite{Bredon1972}[p. 175, 2.1 Theorem]. While the second case holds and the orders of the quotient groups $H_1/P$, $H_2/P$ are odd, we again apply this theorem to state that $\Sigma^{H_1}$ and $\Sigma^{H_2}$ are orientable, since $$\Sigma^{H_i}=\Big(\Sigma^P\Big)^{H_i/P}.$$
	Therefore we have a well-defined intersection number of $C(H_1)$ and $C(H_2)$ in $\Sigma^P$. From the proof of orientability of $\Sigma^P$ (\ref{l:2}), we deduce that all the homology groups $H_k(\Sigma^P)$ are finite for $k=1,\ldots,\dim\Sigma^P-1$. Since $\Sigma^P$ is orientable, the Poincare duality holds, 
	$
	H^k(\Sigma^P)\cong H_{n-k}(\Sigma^P).
	$
	As result $H^k(\Sigma^P)$ are finite for $k=1,\ldots,\dim\Sigma^P-1$. Denote by $c_1$ and $c_2$ the Poincare duals to $(i_{C(H_1)})_*([C(H_1)])$ and $(i_{C(H_2)})_*([C(H_2)])$ respectively. As the dimensions of $C(H_1)$ and $C(H_2)$ are positive, $c_i$ belongs $H^{k_i}(\Sigma^P)$ for $i=1,2$ with $k_i=1,\ldots,\Sigma^P-1$ and $k_1+k_2=\dim\Sigma^P$. $H^{k_i}(\Sigma^P)$ have only torsion part, therefore $c_1$ and $c_2$ are of finite order and $c_1\cup c_2\in H^{\dim\Sigma^P}(\Sigma^P)$ is zero.
	Hence, from Theorem \ref{t:7} we have
	$$
	C(H_1)\cdot C(H_2)=\langle c_1\cup c_2,[\Sigma^P]\rangle=\langle 0,[\Sigma^P]\rangle=0.
	$$
	On the other hand, by our supposition, $|\Sigma^G|=1$, and therefore the intersection number of $C(H_1)$ and $C(H_2)$ is either $1$ or $-1$, a contradiction.
\end{proof}
\section{Exclusion Algorithm}
In this paragraph we present the exclusion algorithm combining the discussed strategies.

Assume we are given a group $G$ and a positive integer $n$ for which we want to exclude an effective one fixed point action of $G$ on $S^n$.
Throughout, we assume that the postulated fixed point is $x\in S^n$. In order to obtain the exclusion algorithm, we perform the following steps.
\begin{enumerate}
	\item \emph{Determine all faithful real characters of $G$ of dimension $n$}
	\newline
	Since only faithful $\mathbb{R}G$-modules of dimension $n$ can occur as the $\mathbb{R}G$-module structure of $T_xS^n$, we can restrict to faithful real characters of $G$ of dimension $n$. Moreover we can rule out the trivial character since we have an isolated fixed point. We use a combinatoric approach to generate these characters. First, find the partitions of $n$ into summands being the dimensions of real nontrivial irreducibles. Next, for any partition $$n=a_1d_1+\ldots+a_md_m,$$ $d_1<\ldots<d_m$, we arrange all choices of real irreducibles $X_{i,1},\ldots,X_{i,a_i}$ for $i=1,\ldots,m$ such that $$X_{i,j}(1)=d_i$$ for $j=1,\ldots,a_i$. Now, it only remains to verify whether the obtained character, $\sum_{1\leq i\leq m}\sum_{1\leq j\leq a_i}X_{i,j}$, is faithful.
	\item \emph{Determine good subgroups triples}
	\newline
	In this step we look for $H_1, H_2$ - subgroups of $G$ and $P\leq H_1\cap H_2$ satisfying the conditions of either of the strategies - Theorem \ref{strategy:1} or Theorem \ref{strategy:2}, not concerning the dimension (they shall be checked later since these conditions depend on the module - character). In both situations, $\langle H_1, H_2\rangle=G$ and $P$ is a primary group. For the Theorem \ref{strategy:1} we expect $H_1,H_2$ to be mod-primary subgroups. In the case of Theorem \ref{strategy:2}, either the orders of $H_1$ and $H_2$ are odd, or $P$ is normal in both $H_1$ and $H_2$ and the orders of the quotient groups $H_i/P$ are odd for $i=1,2$. We collect the triples of such subgroups $(H_1,H_2,P)$ into two parts corresponding to a given strategy.
	\item \emph{Check the dimension conditions for characters}
	\newline
	For any faithful character generated in the first step, we consider every good subgroup triple for either the strategy. Once we fix such character $X$ and triple $(H_1,H_2,P)$, we check the dimension conditions. If they are satisfied for at least one of the strategies, we conclude that they are an obstacle to one fixed point action. If these obstacles are found for any character to be considered, we can exclude the one fixed point action. 
\end{enumerate}
If the conditions of Theorem \ref{strategy:1} or Theorem \ref{strategy:2} are satisfied, this algorithm provides the proof that there does not exist the postulated action. The algorithm has been implemented in GAP language. 
\newline
\newline
\emph{Note} We can use also a theorem from the paper of Morimoto, \cite{Morimoto2008}[Lemma 2.1], which asserts that if $x_0$ is the only fixed point of the action of $G$ on a $\mathbb{Z}$-homology sphere $\Sigma$ and $K$ is the intersection of all subgroups of $G$ of index $2$, then $\dim T_{x_0}\Sigma^K=0$. This condition restricts essentially the set of "candidating" $\mathbb{R}G$-modules.
\label{s:4}
\section{Results}
\label{s:5}
In introduction, we presented particular Oliver groups and dimensions for which the exclusion strategy ruled out one fixed point actions on spheres of these dimensions. Those results were obtained by GAP computations. The strategy was tested for all Oliver groups of order up to $216$, as well as for $A_5\times C_5$ and $A_5\times C_7$.

Let us see how the strategy works for the particular example of $G=A_5\times C_3$ for $n=6$.

$G$ can be presented as the subgroup of the symmetric group on eight letters, denoted by $S_8$. Namely, $G$ is generated by the following two elements of $S_8$: $\begin{pmatrix}
1 & 5 & 2&4&3 
\end{pmatrix}\begin{pmatrix}
6&8&7
\end{pmatrix}$ and $\begin{pmatrix}
1 & 4 & 2&5&3 
\end{pmatrix}\begin{pmatrix}
6&7&8
\end{pmatrix}$. $G$ has $21$ conjugacy classes of subgroups:
\begin{itemize}
	\item $d_1$ with representative $\{id\}$ isomorphic to the trivial group,
	\item $d_2$ with representative $\langle\begin{pmatrix}2&5\end{pmatrix}\begin{pmatrix}3&4\end{pmatrix}\rangle$ isomorphic to $C_2$,
	\item $d_{3,1},d_{3,2},d_{3,3}$ with representatives $\langle\begin{pmatrix}6&8&7\end{pmatrix}\rangle$, $\langle\begin{pmatrix}2&4&5\end{pmatrix}\rangle$ and 
	\newline
	$\langle\begin{pmatrix}2&4&5\end{pmatrix}\begin{pmatrix}6&8&7\end{pmatrix}\rangle$ respectively, isomorphic to $C_3$,
	\item $d_4$ with representative $\langle\begin{pmatrix}2&4\end{pmatrix}\begin{pmatrix}3&5\end{pmatrix}, \begin{pmatrix}2&4\end{pmatrix}\begin{pmatrix}3&4\end{pmatrix}\rangle$ isomorphic to 
	\newline
	$C_2\times C_2$,
	\item $d_5$ with representative $\langle\begin{pmatrix}1&3&4&5&2\end{pmatrix}\rangle$ isomorphic to $C_5$,
	\item $d_{6,1}$ with representative $\langle\begin{pmatrix}1&3\end{pmatrix}\begin{pmatrix}2&5\end{pmatrix},\begin{pmatrix}2&5&4\end{pmatrix}\rangle $ isomorphic to $D_6$,
	\item $d_{6,2}$ with representative $\langle\begin{pmatrix}2&5\end{pmatrix}\begin{pmatrix}3&4\end{pmatrix},\begin{pmatrix}6&8&7\end{pmatrix}\rangle $ isomorphic to $C_6$,
	\item $d_9$ with representative $\langle\begin{pmatrix}2&4&5\end{pmatrix},\begin{pmatrix}6&8&7\end{pmatrix}\rangle$ isomorphic to $C_3\times C_3$,
	\item $d_{10}$ with representative $\langle\begin{pmatrix}2&3\end{pmatrix}\begin{pmatrix}4&5\end{pmatrix},\begin{pmatrix}1&2&5&4&3\end{pmatrix}\rangle$ isomorphic to $D_{10}$,
	\item $d_{12,1},d_{12,2},d_{12,3}$ with representatives 
	\newline
	$\langle\begin{pmatrix}2&4&5\end{pmatrix},\begin{pmatrix}2&3\end{pmatrix}\begin{pmatrix}4&5\end{pmatrix},\begin{pmatrix}2&5\end{pmatrix}\begin{pmatrix}3&4\end{pmatrix}\rangle$, 
	\newline
	$\langle\begin{pmatrix}2&4&5\end{pmatrix}\begin{pmatrix}6&8&7\end{pmatrix},\begin{pmatrix}2&3\end{pmatrix}\begin{pmatrix}4&5\end{pmatrix},\begin{pmatrix}2&5\end{pmatrix}\begin{pmatrix}3&4\end{pmatrix}\rangle$ and 
	\newline
	$\langle\begin{pmatrix}2&4&5\end{pmatrix}\begin{pmatrix}6&7&8\end{pmatrix},\begin{pmatrix}2&3\end{pmatrix}\begin{pmatrix}4&5\end{pmatrix},\begin{pmatrix}2&5\end{pmatrix}\begin{pmatrix}3&4\end{pmatrix}\rangle$ respectively, isomorphic to $A_4$,
	\item $d_{12,4}$ with representative 
	\newline
	$\langle\begin{pmatrix}2&4\end{pmatrix}\begin{pmatrix}3&5\end{pmatrix},\begin{pmatrix}2&5\end{pmatrix}\begin{pmatrix}3&4\end{pmatrix},\begin{pmatrix}2&4&5\end{pmatrix}\begin{pmatrix}6&8&7\end{pmatrix}\rangle$ isomorphic to
	\newline
	$C_6\times C_2$,
	\item $d_{15}$ with representative $\langle\begin{pmatrix}1&3&4&5&2\end{pmatrix},\begin{pmatrix}6&8&7\end{pmatrix}\rangle$ isomorphic to $C_{15}$,
	\item $d_{18}$ with representative $\langle\begin{pmatrix}1&3\end{pmatrix}\begin{pmatrix}4&5\end{pmatrix},\begin{pmatrix}2&5&4\end{pmatrix},\begin{pmatrix}6&8&7\end{pmatrix}\rangle$ isomorphic to $D_{6}\times C_3$,
	\item $d_{30}$ with representative 
	\newline
	$\langle\begin{pmatrix}1&5\end{pmatrix}\begin{pmatrix}3&4\end{pmatrix}\begin{pmatrix}6&8&7\end{pmatrix},\begin{pmatrix}1&4\end{pmatrix}\begin{pmatrix}2&5\end{pmatrix}\begin{pmatrix}6&8&7\end{pmatrix}\rangle$ isomorphic to
	\newline
	$D_{10}\times C_3$,
	\item $d_{36}$ with representative $\langle\begin{pmatrix}2&4&5\end{pmatrix},\begin{pmatrix}2&4\end{pmatrix}\begin{pmatrix}3&5\end{pmatrix},\begin{pmatrix}6&8&7\end{pmatrix}\rangle$ isomorphic to $A_4\times C_3$,
	\item $d_{60}$ with representative $\langle\begin{pmatrix}1&3&4&5&2\end{pmatrix},\begin{pmatrix}1&4&3&5&2\end{pmatrix}\rangle$ isomorphic to $A_5$,
	\item $d_{180}$ with the only representative isomorphic to $G$.
\end{itemize}
The following table shows the fixed point dimensions of the characters for all subgroups of $G$:
\begin{center}
	\begin{tabular}{l|l*{9}{c}r}
		&  $d_{1}$& $d_{2}$ & $d_{3,1}$ & $d_{3,2}$ & $d_{3,3}$  & $d_{4}$ & $d_{5}$&$d_{6,1}$ & $d_{6,2}$&$d_9$&$d_{10}$  \\
		\hline
		$X_{2}$ & 2 & 2 & 0 & 2 & 0 & 2 & 2&2&0&0&2 \\
		$X_{3,1}$   & 3 & 1 & 3 & 1 & 1 & 0 & 1&0&1&1&0 \\
		$X_{3,2}$   & 3 & 1 & 3 & 1 & 1 & 0 & 1&0&1&1&0 \\
		$X_{6,1}$    & 6 & 2 & 0 & 2 & 2 & 0 & 2&0&0&0&0 \\
		$X_{6,2}$     & 6 & 2 & 0 & 2 & 2 & 0 & 2&0&0&0&0\\
		$X_4$     & 4 & 2 & 4 & 2 & 2 & 1 & 0&1&2&2&0 \\
		$X_8$     & 8 & 4 & 0 & 4 & 2 & 2 & 0&2&0&0&0 \\
		$X_5$     & 5 & 3 & 5 & 1 & 1 & 2 & 1&1&3&1&1 \\
		$X_{10}$      & 10 & 6 & 0 & 2 & 4 & 4 & 2&2&0&0&2 \\
	\end{tabular}
\end{center}
\begin{center}
	\begin{tabular}{l|l*{8}{c}r}
		&  $d_{12,1}$& $d_{12,4}$ & $d_{12,2}$ & $d_{12,3}$ & $d_{15}$  & $d_{18}$ & $d_{30}$&$d_{36}$ & $d_{60}$&$d_{180}$\\
		\hline
		$X_{2}$ & 2 & 0 & 0 & 0 & 0 & 0 & 0&0&2&0 \\
		$X_{3,1}$   & 0 & 0 & 0 & 0 & 1 & 0 & 0&0&0&0 \\
		$X_{3,2}$   & 0 & 0 & 0 & 0 & 1 & 0 & 0&0&0&0 \\
		$X_{6,1}$    & 0 & 0 & 0 & 0 & 0 & 0 & 0&0&0&0 \\
		$X_{6,2}$     & 0 & 0 & 0 & 0 & 0 & 0 & 0&0&0&0\\
		$X_4$     & 1 & 1 & 1 & 1 & 0 & 1 & 0&1&0&0 \\
		$X_8$     & 2 & 0 & 0 & 0 & 0 & 0 & 0&0&0&0 \\
		$X_5$     & 0 & 2 & 0 & 0 & 1 & 1 & 1&0&0&0 \\
		$X_{10}$      & 0 & 0 & 2 & 2 & 0 & 0 & 0&0&0&0 \\
	\end{tabular}
\end{center}
Now, we analyze all the possible characters which can occur as characters of tangent module to $S^6$ at the fixed point and exclude the one fixed point actions for them. We have the following characters.
\begin{itemize}
		\item $X=X_{6,i}$ for $i=1,2$; in this case take 
		\begin{eqnarray*}
			H_1&=&\langle\begin{pmatrix}2&5\end{pmatrix}\begin{pmatrix}3&4\end{pmatrix},\begin{pmatrix}6&8&7\end{pmatrix}\rangle\in d_{6,2}, \\
			H_2&=&\langle\begin{pmatrix}1&5&4\end{pmatrix},\begin{pmatrix}6&8&7\end{pmatrix}\rangle\in d_9,\\
			P&=&\langle\begin{pmatrix}6&8&7\end{pmatrix}\rangle\in d_{3,1}
		\end{eqnarray*}
		We see that $P$ is contained in $H_1\cap H_2$. $H_1$ and $H_2$ are  mod-$p$-cyclic and they generate $G$. From the fixed dimension table, we see that $\dim X^P=0$ and we exclude the effective one-fixed point action by the discrete submodules restriction.
		\item $X=X_2+X_4$; we take 
		\begin{eqnarray*}
			H_1&=&\langle\begin{pmatrix}2&4&5\end{pmatrix}\rangle\in d_{3,2},\\
			H_2&=&\langle\begin{pmatrix}1&3&2\end{pmatrix},\begin{pmatrix}6&8&7\end{pmatrix}\rangle\in d_{3,3},\\
			P&=&\{id\}\in d_1
		\end{eqnarray*}
		As in the previous case, $\langle H_1\cup H_2\rangle = G$, so that we can apply the intersection number restriction, since $$\dim X^{H_1}+\dim X^{H_2}=4+2=6=\dim X^P.$$
\end{itemize}
\section{Conclusions}
The exclusion algorithm works well for the case where there are relatively few faithful $\mathbb{R}G$-modules of a given dimension. It is often the case that the big amount of such modules generates situations which do not fit for either strategy. However, for low dimensions, such as from $6$ up to $10$, there are few faithful modules and in many cases, we are able to exclude effective one fixed point actions. Even this is very interesting, since for the current state of knowledge we don't know whether a particular Oliver group can act with one fixed point on some sphere of low dimensions varying from $6$ to $10$. For many groups we were able to exclude effective one fixed point actions on $S^6$. As an example one can give $A_5\times C_p$ for $p=3,5,7$. Since the dimension $7$ was not ruled out, one could try to construct an effective one fixed point action on $S^7$ for these groups. If it was the case, we would know the minimal dimension of a sphere on which $A_5\times C_p$ can act effectively with one fixed point for $p=3,5,7$.
\section*{Acknowledgements}
The authors would like to thank Prof. Krzysztof Pawa\l{}owski for the support and interest in our research. We are also very thankful to participants of the algebraic topology seminar held at Adam Mickiewicz University, for helpful comments on the topic of this paper.

\bibliographystyle{plain}

\end{document}